\documentclass[12pt]{article}
\usepackage{amsmath,amsthm,amssymb}
\usepackage{times}
\usepackage{enumerate}
\usepackage{geometry}
\usepackage{xcolor}
\usepackage{verbatim}
\usepackage{amstext}
\usepackage{enumitem}
\usepackage[affil-it]{authblk}
\theoremstyle{plain}
\newtheorem{thm}{Theorem}[section]

   \theoremstyle{remark}
  \newtheorem{rem}[thm]{\protect\remarkname}

\numberwithin{equation}{section}
\numberwithin{figure}{section}

  \theoremstyle{plain}
  \newtheorem{lemma}[thm]{Lemma}
\theoremstyle{example}
  \newtheorem{exam}[thm]{Example}
  \theoremstyle{plain}
  
  \theoremstyle{plain}
  
 \theoremstyle{plain}
  \newtheorem{definition}[thm]{Definition}

\newtheorem*{xthm}{Theorem}

  \newcounter{casectr}

\def\sqr#1#2{{\,\vcenter{\vbox{\hrule height.#2pt\hbox{\vrule width.#2pt
height#1pt \kern#1pt\vrule width.#2pt}\hrule height.#2pt}}\,}}

\newcommand{\vp}{\varphi}
\usepackage{babel}
  \providecommand{\remarkname}{Remark}

\makeatletter
\def\@maketitle{%
  \newpage
  \null
  \vskip 2em%
  \begin{center}%
  \let \footnote \thanks
    {\Large\bfseries \@title \par}%
    \vskip 1.5em%
    {\normalsize
      \lineskip .5em%
      \begin{tabular}[t]{c}%
        \@author
      \end{tabular}\par}%
    \vskip 1em%
  \end{center}%
  \par
  \vskip 1.5em}
\makeatother

\begin{document}

\title{\huge\textmd{A geometric characterisation of real C$^*$-algebras}}

\author{Cho-Ho Chu}

\affil{\it School of Mathematical Sciences, Queen Mary, University of London\\
 London E1 4NS, UK\\ Email: c.chu@qmul.ac.uk}


\maketitle

\begin{abstract}
We characterise the positive cone of a real C*-algebra geometrically.
Given an open cone $\Omega$ in a real Banach space
$V$, with closure $\overline \Omega$, we show that
$\Omega$  is the interior of the positive cone
 of a  unital real C*-algebra if and only if it is a Finsler symmetric cone
 with an orientable extension, which is equivalent to the condition that
$V$ is, in an equivalent norm, the hermitian part of a unital real C*-algebra
with positive cone $\overline\Omega$.
\end{abstract}

\indent {\footnotesize {\bf Keywords:}  Real C*-algebra. Banach manifold. Finsler symmetric cone. Jordan algebra}\\
\indent {\footnotesize {\bf MSC (2020):} 46L05, 58B20, 22E65, 17C65, 46B40}\\

\section{Introduction}

In recent decades, {\it real} C*-algebras have received increasing interest
due to a number of credible applications,
for example, in the classification of manifolds of positive scalar curvature,
 representation theory and orientifold string theories, as noted in \cite{r}.

 A complex C*-algebra can be regarded as a real C*-algebra
 when the underlying scalar field is restricted to the reals. On the other hand, a real C*-algebra can be
 complexified to a complex C*-algebra. Nevertheless, the larger class of real C*-algebras has some
 extra structures.

 {\it Real C*-algebras} are real Banach *-algebras  isometrically
 *-isomorphic to a norm-closed *-algebra of bounded operators on a real Hilbert space.
 They arise from  involutary $*$-antiautomorphisms $\vp$  of complex C*-algebras $\mathcal{A}$
  in that each of them is  of the form
 $$A= \{a\in \mathcal{A}: \vp(a)=a^*\}$$
 where $\mathcal{A}= A+iA$ and $A \cap iA = \{0\}$. We refer to \cite{g} for more details of
 real C*-algebras. Henceforth, C*-algebras are either real or complex unless specified otherwise.

 Positive elements and functionals are fundamental in the theory of
 C*-algebras. The positive cone resides in the hermitian part of a C*-algebra and the normalised positive
 functionals form the state space of a C*-algebra.
In \cite{aos}, state spaces of unital complex C*-algebras have been characterised geometrically
among compact convex sets,
which is the culmination of a series of works in \cite{as1, as3}.

In this paper, we give a simple geometric characterisation of the positive cones of
real unital  C*-algebras (as well as the complex ones) among positive cones in real Banach spaces.
This is equivalent to characterising
geometrically  the hermitian parts of unital C*-algebras among partially ordered Banach spaces.

By the {\it hermitian part of a C*-algebra} $A$, we mean the partially ordered real Banach space
$A_h :=\{a\in A: a^*=a\}$ consisting of hermitian elements in $A$.

 The aforementioned result of state spaces in \cite{aos} is formulated as a characterisation of the hermitian part
 of a unital complex C*-algebra isomorphically and likewise,
 our result is stated as a characterisation of the hermitian part
 of a real C*-algebra, as follows.

 \begin{xthm} Let $V$ be a real Banach space with  an open cone $\Omega$
and closure $\overline \Omega$. Then $V$  is, in an equivalent norm,
 the hermitian part of a unital real C*-algebra with positive cone $\overline \Omega$ if and only if
 $\Omega$ is a Finsler symmetric cone with an orienatable extension.
 \end{xthm}

 In other words,
 {\it  an open cone $\Omega$ in a real Banach space $V$ is the interior of the positive cone
 of a  unital real C*-algebra if and only if it is a Finsler symmetric cone
 with an orientable extension}.

The preceding assertion is proved in Theorem \ref{rc}. The special case of {\it complex C*-algebras} is proved in
Theorem \ref{c} which asserts, equivalently, that {\it an open cone  $\Omega$ in a real Banach space $V$ is
the interior of the positive
cone of a  unital complex C*-algebra if and only if it is a normal linearly homogeneous Finsler symmetric cone
with an orientation}.

We note that a characterisation of von Neumann algebras in terms of certain assoicated cones has been established by
Connes  in \cite{con}, where it is shown that
a closed cone in a (complex) Hilbert space is the closed cone associated to a cyclic and separating vector
of a von Neumann algebra if and only if it is self-dual, facially homogeneous and orientable.

It is interesting to observe from our results the intimate relationship between the C*-algebraic structures and
the Finsler geometric structures of cones. In fact, they determine each other.
How does one construct a C*-product in $V$, given a  Finsler symmetric
 structure of the cone? A C*-product
 $$ab = \frac{1}{2}(ab+ba)  + \frac{1}{2}(ab-ba)$$
is composed of two parts, namely, the Jordan product
$$a\circ b =  \frac{1}{2}(ab+ba)$$
and the Lie product
$$[a, b] = ab-ba.$$
A Finsler symmetric cone $\Omega$ in a real Banach space $V$
is a real symmetric Banach manifold modelled on $V$, which induces a Jordan structure
in $V$. Further, an orientation on $\Omega$, if it exists, then provides the Lie product.
Conversely, the hermitian part $A_h$ of a C*-algebra $A$ carries a natural Jordan algebraic structure,
which gives rise
to a Finsler symmetric structure of the cone, and the inner derivations of $A$ provide an orientation.
These remarks reveal the idea and gist of the proof of Theorem \ref{rc}.

In what follows, we discuss Finsler geometry of cones
and complete the details of the preceding discussion, in which
the notion of an orientation in Definition \ref{o} is adapted with simplification
from the one introduced by
Upmeier  \cite{up} for a holomorphic characterisation of
complex C*-algebras. This notion differs from that of an orientation for a self-dual cone
introduced in \cite[D\'efinition 4.11]{con}.

\section{Geometry of open cones}

Let $V$ be a real Banach space. The Banach algebra $L(V)$ of bounded linear operators on $V$
forms a real Banach Lie algebra in the commutator product
$$[\vp, \psi] = \vp\psi- \psi\vp  \qquad (\vp,\psi \in L(V)).$$
Let $GL(V)$ be the open subgroup of $L(V)$ consisting of invertible elements in $L(V)$.
It is a real Banach Lie group with Lie algebra $L(V)$.

Let $\Omega$ be an open cone in a real Banach space $V$, where a {\it cone}  is defined to be
a non-empty set $\Omega$ satisfying $\Omega + \Omega \subset \Omega$ and
$\alpha \Omega \subset \Omega$ for all $\alpha >0$. Then $\Omega$ is a real  Banach manifold
modelled on $V$. We refer to \cite{book, chu1} for a brief introduction to Banach manifolds
and Banach Lie groups.

The linear maps $g\in GL(V)$ satisfying
$g(\Omega) = \Omega$
 form a subgroup of $GL(V)$ and will be denoted by
 \begin{equation}\label{G}
G(\Omega)=\{g\in GL(V): g(\Omega) = \Omega\}.
\end{equation}
We shall call $G(\Omega)$ the {\it linear automorphism group} of $\Omega$.
An element $g\in GL(V)$ belongs to
$G(\Omega)$ if and only if $g(\overline\Omega) = \overline\Omega$, the latter denotes the closure of $\Omega$.
 Hence $G(\Omega)$ is a closed subgroup of
 $GL(V)$ and can be topologised  to
 a real Banach Lie group with Lie algebra
 \begin{equation}\label{go}
\frak g(\Omega) =\{ X\in L(V): \exp t X \in G(\Omega)~  \forall t \in \mathbb{R}\}
\end{equation}
(cf.\,\cite[p.\,387]{up}). For each $e\in \Omega$, we define
\begin{equation}\label{e}
\frak g (\Omega)_e := \{X\in \frak g (\Omega): X(e) =0\} =
\{X\in \frak g (\Omega): (\exp t X )(e) = e~  \forall t \in \mathbb{R}\}
\end{equation}
which is a closed Lie subalgebra of $\frak g(\Omega)$, and will be called the
{\it derivation algebra induced by} $e$.

\begin{definition}\rm  An open cone $\Omega$ in a real Banach space
is called {\it linearly homogeneous} if the linear automorphism
group $G(\Omega)$ acts transitively on $\Omega$,
that is, given $a,b\in \Omega$, there is a continuous linear isomorphism $g\in G(\Omega)$ such that
$g(a)=b$.
\end{definition}

Let $\Omega$ be an open cone in a real Banach space $V$ with norm $\|\cdot\|$, and let $\leq$ be
the partial order defined by the closure $\overline \Omega$, which is a cone,  so that
$$x\leq y \Leftrightarrow y-x \in \overline\Omega.$$
 In this order, each element $e\in \Omega$ is an {\it order-unit}, that is, for each $v\in V$,  we have
$$-\lambda e \leq v \leq \lambda e$$
for some $\lambda >0$ (cf.\,\cite[p.152]{chu2}).  This implies
\begin{equation*}\label{gen}
V= \Omega -\Omega.
\end{equation*}
 An  order-unit $e\in \Omega$ induces a semi-norm $\|\cdot\|_e$ on $V$,
 defined by
 \[ \|x\|_e = \inf\{\lambda>0: -\lambda e \leq x \leq \lambda e\} \qquad (x\in V)\]
which satisfies
\begin{equation}\label{eq}
-\|x\|_e e\leq x \leq \|x\|_e e.
\end{equation}
Since $\{x\in V: \|x\|_e =0\} = \overline \Omega \cap -\overline\Omega$, the semi-norm $\|\cdot\|_e$ is a norm
if and only if  $$ \overline \Omega \cap -\overline\Omega=\{0\}$$
 in which case, $\Omega$ is called a {\it proper cone}
and $\|\cdot\|_e$ is
called the {\it order-unit norm} induced by $e$.
All order-unit norms induced by elements in $\Omega$ are mutually equivalent.

A linear functional $f: V \longrightarrow \mathbb{R}$  is called {\it positive} if $f(\overline \Omega)
\subset [0, \infty)$. Let $e\in \Omega$.
Denote the {\it state space} of $V$ (with respect to $e$) by
\begin{eqnarray*}\label{s}
S_e(V)&=& \{f\in V^*: f(e)= 1, f \mbox{ is positive} \}\\
&=&  \{f\in (V, \|\cdot\|_e)^*: f(e)= 1=\|f\|_e\}
\end{eqnarray*}
where $\|f\|_e$ denotes the norm of  $f$ with respect to the order-unit norm $\|\cdot\|_e$ of
$V$.
We will make use of the following lemma proved in \cite{chu2}.

\begin{lemma}\label{1}
Let $\Omega$ be a proper open cone in a real Banach space $V$ and let $e\in \Omega$,
which induces an order-unit norm $\|\cdot\|_e$ on $V$.
Then we have
$$\Omega = \bigcap_{f\in S_e(V)} f^{-1}(0, \infty).$$
\end{lemma}

Given a proper open cone $\Omega$ in a real Banach space $V$,
we consider, for later application,  the question of completeness of the order-unit norm $\|\cdot\|_e$
induced by a point $e \in \Omega$. The answer depends on certain order property of the cone.

\begin{definition}\rm An open cone $\Omega$  in a real Banach space $V$ is
called {\it normal} if there is a constant $\gamma >0$ such that
$$0\leq x \leq y \Rightarrow \|x\|\leq \gamma\|y\| \qquad (x,y \in V)$$
where $\leq$ is the partial order defined by the closure $\overline\Omega$.
\end{definition}

The following lemma shows that a normal open cone is proper.

\begin{lemma}\label{norm} Let $\Omega$ be an open cone in a real Banach space $V$ with norm $\|\cdot\|$.
Let $\|\cdot\|_e$ be the semi-norm induced by $e\in \Omega$. The following conditions are equivalent.
\begin{enumerate}
\item[(i)] $\|\cdot\|_e$ is a complete norm.
\item[(ii)] $\Omega$ is normal.
\item[(iii)] $\|\cdot\|_e$ is equivalent to $\|\cdot\|$.
\end{enumerate}
\end{lemma}
\begin{proof} By \cite[Lemma 3.3]{chu2},   $(ii)$  is equivalent to $(iii)$ under the assumption that
$\Omega$ is proper. A close examination of the proof  in  \cite[Lemma 3.3]{chu2}
reveals that this assumption can be removed.

Since $(V, \|\cdot\|)$ is a Banach space, the equivalence of $(i)$ and $(iii)$ follows from the open
mapping theorem and the fact that $\|\cdot\|_e \leq \alpha \|\cdot\|$ for some $\alpha >0$.
To see the latter, observe that there is an open neighbourhood of $e$ contained in the open set $\Omega$.
Hence we can pick $r>0$ so that $\Omega$ contains the open ball $e-B(0,r)$,
 where $B(0,r) = \{x\in V: \|x\| < r\}$. Let $v\in V\backslash \{0\}$.
Then we have $\pm (r/2\|v\|)v \in B(0,r)$
which implies $$ e \mp (r/2\|v\|)v \in e -B(0,r) \subset \Omega,$$ that is,
$$- \frac{2\|v\|}{r} e \leq v \leq \frac{2\|v\|}{r} e$$
proving
$\|v\|_e \leq (2/r)\|v\|$ for all $v \in V$.
\end{proof}

Order structures play an important role in C*-algebras. The interior $\Omega$ of the positive cone
$$\{a\in A:  0\leq a\}$$ of a unital C*-algebra $A$ is a normal open cone containing the identity $e$,
and the order-unit norm $\|\cdot\|_e$
coincides with the C*-norm on the hermitian part
$$A_h =\{a\in A: a^*=a\}$$ of $A$. If $A$ is a real C*-algebra, then its
open cone $\Omega$ has a natural {\it extension} to the open cone $\widetilde \Omega$
of a complex C*-algebra, in the following sense.

\begin{definition}\label{ext}\rm
Given   open cones $\Omega$ and $\widetilde\Omega$ in real Banach spaces $V$ and  $\widetilde V$ respectively,
we say that $(\widetilde V, \widetilde\Omega)$, or simply $\widetilde \Omega$, is an {\it extension} of $\Omega$
if there is a linear isometry $\vp : \widetilde V \longrightarrow \widetilde V$ such that
$$V=\{v\in \widetilde V: \vp(v)=v\}$$
and $\Omega = \widetilde\Omega \cap V$.
\end{definition}

\begin{exam}\label{ex} \rm Let $\Omega$ be the interior of the positive cone of a real C*-algebra $A$ with
identity $e$ and let
$$A= \{ a\in \mathcal{A}: \vp(a) = a^*\}$$
where $\mathcal{A}$ is a complex C*-algebra and $\vp:\mathcal{A}\longrightarrow \mathcal{A}$
is an involutary $*$-antiautomorphism.  Then  $e\in \Omega$ is an order-unit of the hermitian part $A_h$ of $A$
and the order-unit norm  $\|\cdot\|_e$ is the C*-norm on $A_h$.

Let $\widetilde\Omega $ be the interior of the positive cone of $\mathcal{A}$.
Then $(\mathcal{A}_h, \widetilde \Omega)$ is an extension of $\Omega$.
Indeed,  $A_h$ is contained in the
hermitian part $\mathcal{A}_h$ of $\mathcal{A}$ and $\vp$ restricts to a  (real) linear isometry of $\mathcal{A}_h$
such that
$$A_h = \{ a\in \mathcal{A}_h: \vp(a)=a\} \quad {\rm and} \quad \Omega = \widetilde \Omega \cap A_h.$$
The last equality above follows from Lemma \ref{1} since every $f$ in the state space $S_e(\mathcal{A}_h)$
of $\mathcal{A}_h$
restricts to an element in $S_e(A_h)$, and each $f\in S_e(A_h)$ has a norm-preserving extension
$\widetilde f$ to $\mathcal{A}_h$ and  $\widetilde f \in S_e(\mathcal{A}_h)$.
\end{exam}

We now introduce Finsler structures of open cones.
Let $M$ be a Banach manifold (with an analytic structure) modelled on a Banach space $(V, \|\cdot\|_{_V})$,
with  tangent bundle
$$TM = \{(p,v): p\in M, v \in T_p M\}.$$
 A  mapping
$$\nu : TM \longrightarrow [0,\infty)$$ is called a {\it tangent norm} \index{tangent norm}
if $\nu(p, \cdot)$ is a norm on the tangent space $T_p M \approx V$ for each $p\in M$. We call $\nu$
a {\it compatible tangent norm} if it satisfies
the following two conditions.
\begin{enumerate}
\item[(i)] $\nu$ is continuous.
\item[(ii)] For each $p\in M$, there is a local chart $\varphi: \mathcal{U} \rightarrow V$ at $p$, and
constants $0<r<R$ such that
$$ r\|d\varphi(a)(v)\|_{_V} \leq \nu(a,v) \leq  R\|d\varphi(a)(v)\|_{_V} \qquad (a\in \mathcal{U}\
\subset M, v\in T_a M).$$
\end{enumerate}
Following \cite{neeb},  we also call a Banach manifold with a compatible tangent norm
a {\it Finsler manifold}. Finsler manifolds are a generalisation of Riemannian manifolds.

Given a Banach manifold $M$ with a compatible tangent norm $\nu$,
a bianalytic map $f:M \longrightarrow M$ is called  a {\it $\nu$-isometry}  if it satisfies
\begin{equation*}\label{nuisom}
\nu(f(p), df(p)(\cdot)) = \nu(p, \cdot) \quad {\rm for~all} \quad (p, \cdot) \in TM.
\end{equation*}

\begin{definition}\label{2.2}
Let $\Omega$ be an open cone in a real Banach space $V$, equipped with a tangent  norm $\nu$.
We say that $\nu$ is {\it $G(\Omega)$-invariant} if each $g \in G(\Omega)$ is a $\nu$-isometry.
\end{definition}

To introduce the  concept of a Finsler symmetric cone, we begin
with the notion of a {\it symmetry} of a manifold. Let $M$ be a Banach manifold with a
compatible tangent norm $\nu$ and let $p\in M$. A {\it $\nu$-symmetry}
(or {\it symmetry}, if $\nu$ is understood) at $p$ is a $\nu$-isometry
$$s: M \longrightarrow M$$ satisfying the following two conditions:
\begin{enumerate}
\item [(i)] $s$ is involutive, that is, $s\circ s$ is the identity map on $M$,
\item [(ii)] $p$ is an isolated fixed-point of $s$, in other words,
$p$ is the only point in some neighbourhood of $p$ satisfying $s(p)=p$.
\end{enumerate}

By a {\it symmetric Finsler  manifold}, or a {\it symmetric Banach manifold} (cf.\,\cite{up1}),
we mean a {\it connected} Banach manifold $M$, equipped with a compatible tangent norm $\nu$, such that
there is a unique $\nu$-symmetry $s_p: M \longrightarrow M$ at each $p\in M$.

Symmetric Finsler manifolds generalise finite dimensional Riemannian symmetric spaces.

\begin{definition}\rm A {\it Finsler symmetric cone} is an open cone $\Omega$ in a {\it real} Banach space $V$,
 which is a symmetric Finsler manifold equipped with a
$G(\Omega)$-invariant compatible tangent norm $\nu$.
\end{definition}

Let $\Omega$ be a Finsler symmetric cone in a real Banach space $V$ and let
$${\rm Aut}\,\Omega = \{f\in {\rm Diff}(\Omega): f\circ s_p = s_{f(p)}\circ f~ \forall p\in \Omega\}$$
be the automorphism group of $\Omega$,
which is a subgroup of the diffeomorphism group Diff$(\Omega)$ of $\Omega$.
By \cite[Theorem 2.4, Theorem 5.12]{klotz} (see also \cite[Appendix]{chu2}),
Aut\,$\Omega$ carries the structure of a real Banach Lie group, with Lie algebra
\begin{equation}\label{kill}
{\rm Kill}\,\Omega =\{X\in\mathcal{V}(\Omega): \exp tX \in {\rm Aut}\,\Omega~ \forall t \in \mathbb{R}\}
\end{equation}
which is a Banach Lie algebra in some norm and a subalgebra of the
Lie algebra $\mathcal{V}(\Omega)$ of smooth vector fields on $\Omega$. A vector field $X\in \mathcal{V}(\Omega)$
has a (local) representation
$$X =  f\frac{\partial}{\partial x}$$
where $f: \Omega \longrightarrow V$ is a smooth function. We call $X$ a {\it linear vector field}
if $f$ is (the restriction of) a continuous linear map on $V$ in which case, we identify $X$ as an element in
$L(V)$. On the other hand, each element $X\in L(V)$ can be identified naturally as a linear vector
field on $\Omega$.

Pick $e\in \Omega$.
Then the adjoint representation
$$\theta = Ad(s_e): {\rm Kill}\,\Omega \longrightarrow {\rm Kill}\,\Omega$$
of the symmetry $s_e \in {\rm Aut}\, \Omega$
is an involution and the Lie algebra ${\rm Kill}\,\Omega$ has a Cartan decomposition
\begin{equation}\label{kp}
{\rm Kill}\,\Omega = \frak k_e \oplus \frak p_e
\end{equation}
satisfying
\begin{equation}\label{kp1}
\frak k_e =\{X\in {\rm Kill}\, \Omega: X(e)=0\}, \quad [\frak p_e, \frak p_e]\subset \frak k_e
\end{equation}
where $\frak k_e$ is the $1$-eigenspace and $\frak p_e$ the $(-1)$-eigenspace of $\theta$
(cf.\,\cite[p.158]{chu2}).
Hence the continuous and linear evaluation  map
\begin{equation}\label{map}
X\in \frak p_e \mapsto X(e) \in V
\end{equation}
is bijective as $\frak k_e \cap \frak p_e=\{0\}$. We denote its inverse by
\begin{equation}\label{L}
L: a\in V \mapsto  L(a)= \ell_a\frac{\partial}{\partial x} \in \frak p_e
\end{equation}
so that $L(a)(e)=a=\ell_a(e)$.

 Since  $[L(a), L(b)]\in \frak k_e$ for $a, b\in V$ by (\ref{kp1}), we have $[L(a), L(b)](e) =0$, that is,
$$0= d\ell_b(e)(\ell_a(e)) - d\ell_a(e)(\ell_b(e)) =  d\ell_b(e)(a) - d\ell_a(e)(b).$$
It follows that  the  multiplication in $V$ defined by
\begin{equation}\label{prod}
a\circ b:= d\ell_a(e)(b) \qquad \left(a, b\in V, L(a) = \ell_a\frac{\partial}{\partial x}\right)
\end{equation}
is commutative. Since $L$ is linear, the product $a \circ b$  is bilinear and  turns $V$ into a
commutative (but not necessarily associative) algebra
in which case, the derivative $d\ell_a(e) : V \longrightarrow V$ is the (left) multiplication operator
$$
L_a : x\in V \mapsto a\circ x \in V
$$
on the algebra $(V, \circ)$.
We will always assume this algebraic structure of $V$ and retain the above notations
in the in the sequel.

If the Finsler symmetric cone $\Omega$ is linearly homogeneous, then by \cite[p.15]{chu2},
 $L(a)$ is a linear vector field on $\Omega$ and identified as an element in
 $L(V)$.
 In this case, $L(a)$ is the multiplication operator $L_a$ on $V$ since (\ref{prod}) can be
 expressed as
\begin{equation}\label{prod1}
a\circ b =L(a)(b) \qquad (a,b \in V).
\end{equation}

\section{Jordan and ordered structures in C*-algebras}

A {\it Jordan algebra} ${A}$ is a vector space  equipped with a  bilinear product
$$(a,b) \in A \times {A} \mapsto a\circ b \in {A}$$
which is  commutative
and satisfies the {\it Jordan identity}
$$a\circ(b\circ a^2) = (a\circ b)\circ a^2 \qquad (a,b \in {A}).$$
We do not assume associativity of the product although $A$ can be infinite dimensional.
We call $A$ {\it unital} if it has an identity.

A real Jordan algebra $(A, \circ)$ is called a {\it JB-algebra}
 if it is also a Banach space and the norm
satisfies
$$ \|a\circ  b\| \leq \|a\| \|b\|, \quad \|a^2\| =  \|a\|^2, \quad
\|a^2\| \leq \|a^2 + b^2\|$$ for all $a, b \in  A$.
A JB-algebra ${A}$ admits a natural order structure determined by the closed cone
$${A}_+ =\{x^2: x\in {A}\}$$
\cite[Lemma 3.3.5, Lemma 3.3.7]{stormer}. If a JB-algebra $A$ contains an identity $e$,
then $e$ belongs to the interior of $A_+$ and the norm of $A$ coincides with the order-unit
norm $\|\cdot\|_e$ \cite[Proposition 3.3.10]{stormer}.

A C*-algebra $A$ is a Jordan algebra in the Jordan product
$$a \circ b = \frac{1}{2}(ab +ba).$$
The relevance of JB-algebras lies in the fact that the hermitian part
\begin{equation}\label{h}
A_h =\{a\in A: a^* =a\},
\end{equation}
which is real Jordan subalgebra of $A$, forms a JB-algebra and $\{x^2: x\in A_h\}$ is the positive
cone $\{a\in A: 0\leq a \}$ of $A$.

Indeed, we will see that the C*-structure of $A$ is determined by its JB-algebraic structure
together with an orientation of the positive cone.
A crucial fact  is the following theorem.

\begin{thm}\label{fin} Let $V$ be a real Banach space with an open cone $\Omega$.
The following conditions are equivalent.
\begin{enumerate}
\item[(i)] $V$ is a JB-algebra in an equivalent  norm, with an identity $e\in \Omega$ and
$\Omega$ is the interior of $ \{x^2: x \in V\}$.
\item[(ii)] $\Omega$ is a normal linearly homogeneous Finsler symmetric cone.
\end{enumerate}
\end{thm}
\begin{proof} This result has been proved in \cite[Theorem 4.2]{chu2} under the assumption
that the given cone $\Omega$ is proper. This assumption can be removed in view of
Lemma \ref{norm}. For completeness,
we sketch the essence of the proof but suppress the details.

If $V$ is a JB-algebra in an equivalent norm, with identity $e\in \Omega$ and $\Omega$
is the interior of $\{x^2: x\in V\}$,
then the norm of the JB-algebra $V$ is the order-unit norm $\|\cdot\|_e$, and
by Lemma \ref{norm},  $\Omega$
is a proper normal cone.
Following \cite[Theorem 4.2 (iii) $\Rightarrow$ (ii)]{chu2}, one can show that
 $\Omega$ is a linearly homogeneous Finsler
symmetric cone in the tangent norm $\nu(p, v) = \|v\|_p$ for $(p,v) \in \Omega \times V$,
where $\|\cdot\|_p$ is the order-unit norm induced by $p\in \Omega$.

Conversely, if $\Omega$ is a normal linearly homogeneous Finsler symmetric cone,
and
$${\rm Kill}\,\Omega = \frak k_e \oplus \frak p_e$$
is the Cartan decomposition induced by $e \in \Omega$  in (\ref{kp}), with the continuous bijective linear
map $L: V \longrightarrow \frak p_e$ in (\ref{L}), then $V$ is a JB-algebra in the order-unit norm
$\|\cdot\|_e$, with  identity $e$ and the product $a \circ b = L(a)(b)$
defined in (\ref{prod1}), where $\|\cdot\|_e$ is equivalent to the original norm of $V$
by Lemma \ref{norm}. Further, $\Omega$ is geodesically complete and it follows that
$\Omega =\{\exp X(e): X\in \frak p_e\}$, which can be shown to coincide with the interior of $\{x^2: x\in V\}$.
\end{proof}

Closed subalgebras of the JB-algebra $A_h$ in (\ref{h}) for some C*-algebra $A$ are called {\it JC-algebras}.
Not all JB-algebras are of this form. A JC-algebra $B \subset A_h$ is called {\it reversible} if
$$a_1a_2\cdots a_n + a_n\cdots a_2a_1 \in B$$
whenever $a_1, a_2, \ldots, a_n \in B$.  By \cite{cohn}, this condition is equivalent to the condition that
$$a_1a_2 a_3a_4 + a_4a_3a_2a_1 \in B$$
for $a_1, a_2, a_3,a_4 \in B$.

\section{Characterisation of real and complex C*-algebras}

We are now ready to characterise C*-algebras in terms of the geometric structure of their cones.

Let $\Omega$ be an open cone in a real Banach space $V$. Since $V= \Omega -\Omega$,
a positive homogeneous additive map $J: \Omega \longrightarrow W$ to a real vector space $W$
can be extended naturally to
a linear map $J_V : V \longrightarrow W$  by defining
\begin{equation}\label{ext1}
J_V(\omega_1 - \omega_2) = J(\omega_1) - J(\omega_2) \qquad (\omega_1, \omega_2 \in \Omega).
\end{equation}
Indeed, $J_V$ is well-defined since $\omega_1 - \omega_2 = \omega'_1 - \omega'_2$ implies
$\omega_1 + \omega'_2 =\omega'_1 + \omega_2$ and $J(\omega_1) + J(\omega'_2)
=J(\omega'_1) + J(\omega_2)$.

\begin{definition}\label{o}\rm
Let $\Omega$ be a Finsler symmetric cone in a real Banach space $V$. An {\it orientation}
of $\Omega$ is a continuous positive homogeneous additive map
$$J: \Omega \longrightarrow \frak g(\Omega)_e$$ satisfying
\begin{equation}\label{ori2}
J_V(J(a)(b)) = [d\ell_b(e), d\ell_a(e)] \qquad (a, b\in \Omega)
\end{equation}
where $e\in \Omega$ and $\frak g(\Omega)_e\subset L(V)$ is the derivation algebra defined in (\ref{e}),
with the derivatives $d\ell_a(e), d\ell_b(e)\in L(V)$ given in (\ref{prod}).
\end{definition}

If we assume the algebraic structure of $V$ defined in
(\ref{prod}), then (\ref{ori2}) can be written as
$$J_V(J(a)(b)) = [L_b, L_a] \qquad (a, b\in \Omega)$$
where $d\ell_a(e)=L_a$ and
$ d\ell_b(e)=L_b$ are the multiplication operators on the algebra $(V,\circ)$.

Given a Finsler symmetric cone $\Omega$ and $a,b, e\in \Omega$, we have $[L(a), L(b)](e)=0$,
following a remark after  (\ref{L}). Hence $[L(a), L(b)] \in \frak g (\Omega)_e$ and one can define a continuous positive
homogeneous  map $J: \Omega \longrightarrow \frak g(\Omega)_e$ by
$$J(a) = [L(a), L(b)] \qquad (a\in \Omega).$$
However, $J$ need not be an orientation. In fact, we shall see that the existence of orientation depends on
additional structure of the cone $\Omega$.

\begin{lemma}\label{j1} Let $J: \Omega \longrightarrow {\frak g(\Omega)_e}$ be an orientation of a
Finsler symmetric
cone $\Omega$. Then $J$ extends to a continuous linear map
$$J_V: V \longrightarrow { \frak g(\Omega)_e}$$ satisfying
$$ J_V( J_V(a)(b)) = [d\ell_b(e), d\ell_a(e)] \qquad (a, b\in V).$$
\end{lemma}
\begin{proof}
Let $ J_V: V \longrightarrow { \frak g(\Omega)_e}$  be the extension of $J$ defined in (\ref{ext1}).
Given $a=a_1-a_2$ and $b=b_1-b_2$ in $V$, with $a_1, a_2, b_1, b_2 \in \Omega$, a direct computation
yields
\begin{eqnarray*}
[d\ell_b(e), d\ell_a(e)] &= &[d\ell_{b_1}(e) - d\ell_{b_2}(e), d\ell_{a_1}(e) - d\ell_{a_2}(e)]\\
& =&J_V(J(a_1)(b_1)) - J_V(J(a_2)(b_1))-J_V(J(a_1)(b_2)) +J_V(J(a_2)(b_2))\\
&=& J_V( J(a_1)(b)) - J_V( J(a_2)(b)) =J_V(J_V(a)(b)).
\end{eqnarray*}

It remains to show continuity of $J_V$. Let $(a_n)$ be a sequence in $V$ converging to $a\in V$.
We show that $(J_V(a_n))$ converges to $J_V(a)$ in $\frak g(\Omega)_e$.

Since $e\in \Omega$ and  $\Omega$ is open, there is an open ball $B(0,r)$ centred at $0$ with radius $r>0$
such that $B(0,r) \subset e-\Omega$. We have $\pm \lambda a \in B(0,r)$ for some $\lambda >0$,
which implies $e\pm \lambda a \in \Omega$ and hence $e\pm \lambda a_n \in \Omega$ from some $n$ onwards.
Hence $\lim_n J(e\pm \lambda a_n) = J(e\pm \lambda a)$, which gives
$$J_V(2\lambda a) = J(e+ \lambda a) - J(e - \lambda a) = \lim_{n\rightarrow\infty}
(J(e + \lambda a_n) - J(e- \lambda a_n)) = \lim_{n\rightarrow\infty} J_V(2\lambda a_n),$$
proving $J_V(a_n) \rightarrow J_V(a)$ as $n \rightarrow \infty$.
\end{proof}

Given a linearly homogeneous Finsler symmetric cone $\Omega$ in a real Banach space $V$,
we show below that the existence of
an orientation
\begin{equation*}\label{je}
J: \Omega \longrightarrow \frak g(\Omega)_e
\end{equation*}
of $\Omega$  does not depend on the choice of $e\in \Omega$,
in the sense that if there is one such orientation $J$, then
for each $u\in \Omega$, there is an orientation $J': \Omega \longrightarrow \frak g(\Omega)_u$.
To see this, we first compare the Cartan decompositions induced by $e$ and $u$.

Linear homogeneity entails the existence of $g\in G(\Omega) \subset GL(V)$ such that $$g(u)=e.$$
We note that $G(\Omega) \subset {\rm Aut}\,\Omega$ and $\frak g(\Omega) \subset {\rm Kill}\,\Omega$
by \cite[p.157]{chu2}.
By uniqueness of symmetries,
the symmetry $s_u : \Omega \longrightarrow \Omega$ at $u\in \Omega$ is given by
$$s_u = g^{-1}s_e g \in {\rm Aut}\,\Omega$$
and therefore the adjoint map $Ad(s_u): {\rm Kill}\,\Omega \longrightarrow {\rm Kill}\,\Omega$ is given by
$$Ad(s_u) = Ad(g^{-1})\theta Ad(g)$$
where $\theta = Ad(s_e)$. By (\ref{kp}), we have the Cartan decompositions
$${\rm Kill}\,\Omega = \frak k_e \oplus \frak p_e = \frak k_u \oplus \frak p_u$$
where the vector fields in $\frak p_e$ and $\frak p_u$ are linear, as remarked before.

Observe that $Ad(g^{-1})\frak p_e =\frak p_u$. Indeed $X\in \frak p_e$ if and only if $\theta X=-X$,
which is equivalent to
$$Ad(s_u)Ad(g^{-1})X= Ad(g^{-1})\theta X = -Ad(g^{-1})X.$$
Given $X\in \frak g(\Omega)_e$, we have $Ad(g^{-1})X \in \frak g(\Omega)_u$ since
\begin{equation}\label{u}
\exp t Ad(g^{-1})X (u) = g^{-1}\exp t X g(u) = g^{-1}\exp tX(e) =g^{-1}(e) =u
\end{equation}
for all $t\in \mathbb{R}$.

\begin{lemma}\label{eu}
Let $\Omega$ be a linearly homogeneous Finsler symmetric cone in a real Banach space $V$
and let $J:\Omega \longrightarrow \frak g(\Omega)_e$ be an orientation where $e\in \Omega$.
Then for each $u\in \Omega$, there is an orientation  $J': \Omega \longrightarrow \frak g(\Omega)_u$.
\end{lemma}
\begin{proof} By linear homogeneity, there is a map $g\in G(\Omega)$ such that $g(u)=e$.
We show that  the map    $J': \Omega \longrightarrow \frak g(\Omega)_u$ defined by
$$J'(a) = g^{-1}J(g(a)) g \qquad (a\in \Omega)$$
is an orientation, where $g: V\longrightarrow V$ is a continuous linear isomorphism and
we have $J'(a) =   Ad(g^{-1}) J(g(a)) \in \frak g (\Omega)_u$ from (\ref{u}).

Evidently, $J'$ is continuous, positive homogeneous and additive. Let
$${\rm Kill}\, \Omega = \frak k_u \oplus \frak p_u$$
be the Cartan decomposition induced by the symmetry $s_u$, and let $$\mathcal{L}: V \longrightarrow \frak p_u$$
be the linear isomorphism such that
$$\mathcal{L}(a)(u)=a \qquad (a\in V) $$ as in (\ref{L}).
It remains to show
\begin{equation}\label{j'}
J'_V(J'(a)(b)) = [\mathcal{L}(b), \mathcal{L}(a)] \qquad (a, b\in \Omega)
\end{equation}
in view of (\ref{prod}) and (\ref{prod1}). By assumption, we have
$$J_V(J(a)(b)) = [{L}(b), L(a)] \qquad (a, b\in \Omega)$$
where $L: V \longrightarrow \frak p_e$ is the linear isomorphism in (\ref{L}) with respect to the
Cartan decomposition ${\rm Kill}\,\Omega = \frak k_e \oplus \frak p_e$.

Given $a\in \Omega$, we have
$$Ad(g^{-1})L(g(a))(u) = g^{-1}L(g(a))g(u) = g^{-1}L(g(a))(e)= g^{-1}g(a) =a$$
which gives $\mathcal{L}(a)=Ad(g^{-1})L(g(a))$. It follows that
\begin{eqnarray*}
J'_V(J'(a)(b)) &=& g^{-1}J_V(g J'(a)(b)) g= g^{-1}J_V(J(g(a))g(b))g\\
&=& g^{-1}[L(g(b)), L(g(a))]g = Ad(g^{-1})[L(g(b)), L(g(a))]\\
&=& [Ad(g^{-1})L(g(b)), Ad(g^{-1})L(g(a))] = [\mathcal{L}(b), \mathcal{L}(a)]
\end{eqnarray*}
which proves (\ref{j'}).
\end{proof}

A linear map $\delta: V \longrightarrow V$  on a Jordan algebra $(V,\circ)$ is called a {\it derivation} if it satisfies
$$\delta(a\circ b) = \delta(a)\circ b + a\circ \delta(b) \qquad (a,b\in V).$$
The continuous derivations of a JB-algebra $V$ form a closed Lie subalgebra aut$\,V$ of $L(V)$,
which is the Lie algebra of the Banach Lie group
$${\rm Aut}\, V =\{\vp \in GL(V): \vp (x\circ y) = \vp(x)\circ \vp(y)\}$$
consisting of continuous Jordan automorphisms of $V$.

\begin{rem}\label{deri} On a commutative C*-algebra $A$, the Jordan product $\circ$ coincides with the
underlying associative product. In fact, every derivation on $A$ vanishes (cf.\,\cite[Lemma 4.1.2]{sakai}
\end{rem}

\begin{lemma}\label{j2} Let $\Omega$ be a normal linearly homogeneous Finsler symmetric cone in a real Banach space $V$.
Let $e\in \Omega$ and $V$ carry the JB-algebra structure in the order-unit norm $\|\cdot\|_e$.
Then $\frak g(\Omega)_e$ is   the Lie algebra
aut\,$V$ of continuous derivations of $V$.
\end{lemma}
\begin{proof}
We have $X\in \frak g (\Omega)_e \subset L(V)$ if and only if $\exp tX (e) =e$ for all $t\in \mathbb{R}$,
where $\exp t X \in GL(V)$ and $\exp t X(\Omega) = \Omega$.
The latter condition says that both $\exp tX$ and its inverse are positive linear automorphisms of the JB-algebra $V$,
preserving the identity $e$, for all $t\in \mathbb{R}$. It follows from  \cite[Theorem]{chu3}  that
$\exp tX$ is a Jordan automorphism of $V$ for all $t\in \mathbb{R}$.
This proves that $X \in \frak g(\Omega)_e$ if and only if $\exp tX \in {\rm Aut}\, V$ for all $t\in \mathbb{R}$, that is,
$X \in {\rm aut}\,V$.
\end{proof}

\begin{rem}\label{ori} In view of Lemmas \ref{j1} and \ref{j2}, and (\ref{prod1}),
an orientation $J: \Omega \longrightarrow \frak g(\Omega)_e$ of a {\it normal  linearly homogeneous} Finsler symmetric cone $\Omega$ in a real Banach space $V$  can be expressed as a continuous linear map
\begin{equation}\label{ori1}
J: V \longrightarrow {\rm aut}\, V
\end{equation}
satisfying
$$J(J(a)(b)) = [L_b, L_a] \qquad (a, b\in V)$$
where $V$ carries the structure of a JB-algebra with identity $e\in \Omega$ by Theorem \ref{fin}.
\end{rem}

We derive some properties of the orientation $J$ in (\ref{ori1}), where $\Omega$
is a  normal  linearly homogeneous Finsler symmetric cone in $V$, and
$V$ is a JB-algebra with an identity $e\in \Omega$ and the order-unit norm $\|\cdot\|_e$.
The {\it centre} $Z(V)$
of $V$ is defined by
$$Z(V)= \{z\in V: [L_z,L_x] =0~ \forall x\in V\}=\{z\in V:  z\circ(x\circ y)= x\circ (z\circ y) ~ \forall x,y\in V\}$$
which contains $e$ and is an associative JB-algebra isometrically isomorphic to the abelian real C*-algebra
$C(K)$ of real continuous functions on a compact Hausdorff sapce $K$ \cite[Theorem 3.2.2]{stormer}.

Given a derivation $\delta : V \longrightarrow V$, a direct computation reveals that $\delta (Z(V)) \subset Z(V)$.
Hence $\delta$ restricts to a derivation on $Z(V)$ and by Remark \ref{deri}, we have $\delta(Z(V)) = \{0\}$.

\begin{lemma}\label{cl} Let $\Omega$ be a normal linearly homogeneous Finsler symmetric cone in a real
Banach space $V$, which carries the structure of a JB-algebra with identity $e\in \Omega$ and
centre $Z(V)$.
Let  $J: V \longrightarrow {\rm aut}\, V = \frak g(\Omega)_e$ be an orientation. Then we have
\begin{enumerate}
\item[(i)] $J(a)(b) = - J(b)(a) \qquad (a,b \in V)$.
\item[(ii)] $J^{-1}(0)= Z(V)$.
\item[(iii)] $[L_a, L_b] = [J(b), J(a)] \qquad (a,b\in V)$.
\end{enumerate}
\end{lemma}
\begin{proof}$(i)$ We first show $J^{-1}(0) \subset Z(V)$.
Let $J(z)=0$. Then we have $[L_x, L_z]= J(J(z)(x)) = 0$ for all $x\in V$. Hence $z\in Z(V)$.

Now let $a\in V$. Then $J(J(a)(a)) = [L_a, L_a] =0$ implies  $J(a)(a) \in Z(V)$.
It follows that the closed subalgebra $A(e,a,J(a)(a))$ in $V$ generated by $e$, $a$ and $J(a)(a)$ is
an associative unital JB-algebra, which is isometrically isomorphic to the commutative real C*-algebra
$C(K)$ of real continuous functions on a compact Hausdorff sapce $K$ \cite[Theorem 3.2.2]{stormer}.

Moreover, $J(a)$ restricts to a derivation on $A(e,a,J(a)(a))$, where $J(a)(J(a)(a))=0$ since $J(a)(a)\in Z(V)$.
By Remark \ref{deri}, $J(a)$ vanishes on $A(e,a,J(a)(a))$ and in particular, $J(a)(a)=0$.
As $a\in V$ was arbitrary, we have shown $$J(a+b)(a+b) =0 \qquad (a,b\in V)$$ which yields
$J(a)(b)= -J(b)(a)$

$(ii)$ It suffices to show $Z(V) \subset J^{-1}(0)$.
 Let $z\in Z(V)$. Then for each $x\in V$, the derivation $J(x): V \longrightarrow V$ vanishes on
 $Z(V)$ by a remark above. Hence $(i)$ implies $J(z)(x)= -J(x)(z) =0$ for all $x\in V$, that is, $J(z)=0$.

$(iii)$ Let $a,b\in V$. For each $x\in V$, we have
\begin{eqnarray*}
&& [J(b), J(a)](x) = J(b)J(a)(x) - J(a)J(b)(x)\\
&=& -J(J(a)(x))(b) + J(J(b)(x))(a) \quad ({\rm by} (i))\\
&=& -[L_x, L_a](b) + [L_x, L_b](a)\\
&=& L_aL_x(b) - L_xL_a(b) + L_xL_b(a) - L_bL_x(a)\\
&=& L_aL_b(x) -L_bL_a(x) \quad (\mbox{by commutativity of Jordan product})\\
&=& [L_a, L_b](x).
\end{eqnarray*}
\end{proof}

Let $V$ be a JB-algebra. Then its complexification $V_c:=V+iV$ admits a natural Jordan product,
with an involution $^*$ defined by
$$(x+iy)^* := x-iy \qquad (x,y\in V).$$
By \cite[Theorem 2.8]{w}, the norm of $V$ can be extended to a norm on $V_c$ so that $V_c$
carries the structure of a {\it JB*-algebra}, which means that
 $V_c$ is a complex Jordan algebra as well as a Banach space, equipped with an involution
$*$, such that
$$\|a\circ b\|\leq \|a\|\|b\|, \quad \|a^*\|=\|a\|, \quad \|2a\circ(a\circ a^*)-a^2\circ a^*\|=\|a\|^3 \quad (a,b \in V_c).$$
A complex C*-algebra $A$ is a JB*-algebra in the Jordan product
$$a \circ b = \frac{1}{2}(ab+ba) \qquad (a,b \in A).$$
We refer to \cite[p.174]{book} for more details. A holomorphic characterisation of unital JB*-algebras
has been given in \cite{bku}, which provides a basis for Upmeier's
characterisation in \cite{up} of the open unit balls of complex unital C*-algebras among bounded domains
in complex Banach spaces.

We now characterise  the positive cones of complex unital C*-algebras,
which will be extended to real C*-algebras.

\begin{thm} \label{c} Let $V$ be a real Banach space with an open cone $\Omega$
and closure $\overline\Omega$.
The following conditions are equivalent.
\begin{enumerate}
\item [(i)] $V$, equipped with an equivalent norm, is the hermitian part of a complex unital C*-algebra
with positive cone $\overline \Omega$.
\item[(ii)] $\Omega$ is a normal linearly homogeneous Finsler symmetric cone with an orientation.
\end{enumerate}
\end{thm}

\begin{proof}
$(i) \Rightarrow (ii)$. Let $|\cdot|$ be an equivalent norm on $V$ such that $(V,|\cdot|)$ is the hermitian part $A_h$
of a complex C*-algebra $A$ with identity $e$ and
$$\overline \Omega = \{a\in A: 0\leq a\}= \{a\in A_h: 0\leq a\}.$$
Then $e\in \Omega$ and $|\cdot|=\|\cdot\|_e$. In particular, $(V, \|\cdot\|_e)$ is a unital JB-algebra in the
Jordan product
$$a\circ b = \frac{1}{2}(ab+ba) = L_a(b) \qquad (a, b\in V=A_h)$$
 and by
Theorem \ref{fin}, $\Omega$ is a normal linearly homogeneous Finsler symmetric cone.

We now define an orientation on $\Omega$. By Lemma \ref{j2}, $\frak g(\Omega)_e$ coincides with the
Lie algebra aut\,$V$ of continuous derivations of the Jordan algebra $(V, \circ)$.
Let $J: V \longrightarrow {\rm aut}\, V$ be defined by
$$J(a)(x) =\frac{i}{2}(ax-xa) \qquad (a, x\in V).$$
Then we have
$$ J(J(a)(x))(y) = \frac{1}{4}(xay +yax -axy-yxa) = [L_x, L_a](y) \qquad (a,x,y\in V).$$
Hence $J$ is an orientation on $\Omega$ by Remark \ref{ori}.

$(ii) \Rightarrow (i)$. Let $\Omega$ be a normal linearly homogeneous Finsler symmetric cone
 with an orientation
$J: \Omega \longrightarrow \frak g(\Omega)_e $, where $e\in \Omega$.
 By Theorem \ref{fin}, $V$ is a JB-algebra with identity $e$, in the order-unit norm $\|\cdot\|_e$
 which is equivalent to the original norm of $V$. We have $\frak g(\Omega)_e ={\rm aut}\, V$
 by Lemma \ref{j2}. By Remark \ref{ori}, the orientation has an extension to $V$, denoted by
$J: V \longrightarrow {\rm aut}\, V$.

By $(ii)$ and $(iii)$ in Lemma \ref{cl}, the orientation $J$ induces a continuous linear map
$$\widetilde J : V/Z(V) \longrightarrow {\rm aut}\,V$$ on the quotient $V/Z(V)$, defined by
$$\widetilde J(a+ Z(V) ) := J(a) \qquad (a\in V),$$
which satisfies
$$[\widetilde J(a+Z(V)), \widetilde J(b+Z(V))] = [L_b,L_a]
= \widetilde J(\widetilde J(a+Z(V))(b)+Z(V)) \qquad (a,b\in V).$$
Therefore $\widetilde J$ is an orientation on $V$ in the sense of \cite[Theorem 4.7]{up}.

As noted before, the complexification $V_c=V+iV$ of $(V, \|\cdot\|_e)$ carries the structure of a JB*-algebra
with identity $e$.
Now it follows from \cite[Theorem 5.19]{up} that  the product
$$ab:= L_a(b) - i J(a)(b) \in V_c \qquad (a,b\in V)$$
extends to an associative product on $V_c$, with which $V_c$ becomes a complex C*-algebra
and $V=\{a\in V_c: a^*=a\}$ the hermitian part of $V_c$.
\end{proof}

\begin{definition} \rm Let $\Omega$ be a Finsler symmetric  cone in a real Banach space $V$. An extension
$(\widetilde V, \widetilde\Omega)$
 of $\Omega$, as defined in Definition \ref{ext},
 is called {\it orientable} if $\widetilde\Omega$ is a normal linearly homogeneous
 Finsler symmetric cone in $\widetilde V$ with an orientation
 $J: \widetilde V \longrightarrow {\rm aut\, \widetilde V}$ satisfying
 $$\vp(J(a)(b) ) = J(\vp(b))(\vp(a))  \qquad (a, b\in \widetilde\Omega)$$
where $\vp: \widetilde V \longrightarrow \widetilde V$ is a linear isometry such that $V=\{v\in \widetilde V: \vp(v)=v\}$.
 \end{definition}

\begin{thm}\label{rc} Let $V$ be a real Banach space with an open cone $\Omega$
and closure $\overline\Omega$.
The following conditions are equivalent.
\begin{enumerate}
\item[(i)] $V$, with an equivalent norm, is the hermitian part of a real unital C*-algebra with positive cone
$\overline \Omega$.
\item[(ii)] $\Omega$ is a Finsler symmetric cone with an orientable extension.
\end{enumerate}
\end{thm}
\begin{proof}
$(i) \Rightarrow (ii)$. Let $|\cdot|$ be an equivalent norm so that $(V, |\cdot|)$ is the hermitian part
$A_h$ of a real C*-algebra $A$ with identity $e$, and $\overline \Omega =\{a\in A: 0 \leq a\} \subset A_h$.

As $ \Omega$ is open and convex, it is the interior of $\overline \Omega$ and $e\in \Omega$.
We note that $V=A_h$ is a unital JB-algebra in the Jordan product
$$a\circ b= \frac{1}{2}(ab+ba) \qquad (a,b \in V)$$
and the norm $|\cdot|$ is the order-unit norm $\|\cdot\|_e$ defined by $e$. By Theorem \ref{fin},
$\Omega$ is a normal linearly homogeneous Finsler symmetric cone.

It remains to show that $\Omega$ has an orientable extension.
The complexification $\mathcal{A}= A+iA$ can be equipped with  a complex C*-algebra structure
in which the norm and involution extend those on $A$,  and the map
 $\psi: \mathcal{A}\longrightarrow \mathcal{A}$ defined by
$$\psi(a+ib) = a^* + i b^* \qquad (a, b\in A)$$
is an involutary $*$-antiautomprhism such that
 $$A=\{a\in \mathcal{A}: \psi(a)= a^*\}.$$
Let
$$\widetilde V:= \mathcal{A}_h = \{a+ib \in \mathcal{A}: a^*=a, b^*=-b\}$$
be the hermitian part of $\mathcal{A}$. Then
$\psi(\mathcal{A}_h) = \mathcal{A}_h$ and the restriction
$$\vp = \psi|_{\mathcal{A}_h}: \mathcal{A}_h \longrightarrow \mathcal{A}_h$$ is a linear isometry.
Observe that
$$V=A_h =\{a\in \mathcal{A}_h : \vp(a)=a\}$$
which is a closed Jordan subalgebra of $\mathcal{A}_h$.
Since $\mathcal{A}_h$ is a JB-algebra with identity $e$,
 the interior $\widetilde\Omega$ of the positive cone
$\{a\in \mathcal{A}_h: 0 \leq a\}\subset \widetilde V$ contains $e$
and  is a normal linearly homogeneous Finsler symmetric cone in
$\widetilde V$, by Theorem \ref{fin} again.

We have, from Example \ref{ex},
$$\Omega = \widetilde \Omega \cap A_h$$
and $(\widetilde V, \widetilde\Omega)$ is an extension of $\Omega$.

Finally, we show that $ \widetilde \Omega $ is an orientable extension of $\Omega$.
For each $a\in  \widetilde \Omega $, define
$$J(a): \mathcal{A}_h \longrightarrow \mathcal{A}_h$$ by
$$J(a)(b) =  i(ab-ba) \qquad (b \in \mathcal{A}_h).$$
It can be verified readily that $J(a)$ is a  derivation of the Jordan algebra $\mathcal{A}_h$ and
the map $$J:  \widetilde \Omega \longrightarrow {\rm aut}\, \mathcal{A}_h$$
is positive homogeneous, additive and continuous.
For $a, b \in \widetilde\Omega$, we have
$$\vp J(a)(b) = \psi (i(ab-ba) )= {i} \psi(ab-ba) = {i}(\psi(b)\psi(a)- \psi(a)\psi(b)) =J(\vp(b))(\vp(a)).
$$
This  proves that $ \widetilde \Omega $ is an orientable extension of  $\Omega $.

$(ii) \Rightarrow (i)$. Let $ (\widetilde V,\widetilde \Omega) $ be an orientable extension of $\Omega $
with
$$  \Omega = \widetilde \Omega \cap V \quad {\rm and} \quad V=\{v\in \widetilde V: \vp(v)=v\}$$
where $\vp : \widetilde V \longrightarrow  \widetilde V$ is a linear isometry.

By assumption,
$ \widetilde \Omega $ is a normal linearly homogeneous Finsler symmetric cone with an orientation
$J:  \widetilde V \longrightarrow {\rm aut}\, \widetilde V$ satisfying
$$\vp(J(x)(y) )= J(\vp(y))(\vp(x)) \qquad (x,y \in \widetilde\Omega)$$
where
by Lemma \ref{eu}, we may assume aut$\,\widetilde V = \frak g(\Omega)_e$ with $e\in \Omega \subset  \widetilde \Omega $
so that $ \widetilde V$ is a JB-algebra with identity $e$ and the order-unit norm $\|\cdot\|_e$,
which is equivalent to  the original norm of
$\widetilde V$. By Theorem \ref{c}, the
complexification
$$\mathcal{A}:=\widetilde V + i \widetilde V$$
 is a complex C*-algebra with identity $e$
and hermitian part $ \widetilde V$,
in which the norm extends
$\|\cdot\|_e$ and the product is given by
$$ab:= L_a(b) - i J(a)(b) \qquad (a, b \in  \widetilde V + i  \widetilde V)$$
where $L_a(b) = (ab+ba)/2$ is the Jordan product in the JB-algebra $ \widetilde V $
and $$\overline{\widetilde \Omega}=\{ v\in \widetilde V: 0\leq v\}.$$

Since $\vp:  \widetilde V\longrightarrow \widetilde V$ is a linear isometry
 and $\vp(e) =e$, it is a Jordan automorphism of $\widetilde V$ (cf.\,\cite[Corollary]{chu3}).
 Hence
we have $$\vp(L_a(b)) = L_{\vp(a)}(\vp (b)) \qquad (a, b \in \widetilde V).$$
Further,  we can extend $\vp$
 to a Jordan automorphism $\psi$ of $\mathcal{A}=\widetilde V + i \widetilde V$
 by defining
$$\psi (a+i b) = \vp(a) + i \vp(b) \qquad (a, b \in \widetilde V).$$

We note that $V$ is a closed Jordan
subalgebra of  $\widetilde V$ since $a,b\in V$ implies
$$\vp(L_a(b)) = L_{\vp(a)}(\vp (b)) = L_a(b) $$
and hence $L_a(b) \in V$.
In particular, $V$ is a JC-algebra. Further, $V$ is reversible in $ \widetilde V$. Indeed,
given $x,y\in \widetilde\Omega$, we have
\begin{eqnarray*} \psi(xy) &=& \psi(L_x(y) - i J(x)(y) ) = \vp(L_x(y)) - i\vp( J(x)(y) )\\
&=& L_{\vp(y)}(\vp(x)) - iJ(\vp(y))(\vp(x))\\
&=& \vp(y)\vp(x) = \psi(y)\psi(x)
\end{eqnarray*}
Since $\widetilde V = \widetilde\Omega -\widetilde\Omega$ and $\mathcal{A}= \widetilde V + i\widetilde V$,
we have $ \psi(xy)=\psi(y)\psi(x)$ for all $x,y\in \mathcal{A}$.
Hence, for $a_1, a_2, a_3, a_4 \in V$, we have $a_1a_2a_3a_4 + a_4a_3a_2a_1\in \widetilde V$ and
\begin{eqnarray*}
&&\vp(a_1a_2a_3a_4 + a_4a_3a_2a_1) = \psi(a_1a_2a_3a_4) + \psi(a_4a_3a_2a_1)\\
&=& \psi(a_4)\psi(a_3)\psi(a_2)\psi(a_1) + \psi(a_1)\psi(a_2)\psi(a_3)\psi(a_4)\\
&=& \vp(a_4)\vp(a_3)\vp(a_2)\vp(a_1) + \vp(a_1)\vp(a_2)\vp(a_3)\vp(a_4)\\
&=&  a_4a_3a_2a_1 + a_1a_2a_3a_4
\end{eqnarray*}
that is, $a_1a_2a_3a_4 + a_4a_3a_2a_1 \in V$.

It follows from \cite[Remark 2.5]{storm} that the closed real subalgebra $\mathcal{R}(V)$ generated by $V$
in $\mathcal{A}$
is a real C*-algebra with identity $e$ and $V= \mathcal{R}(V)_h$ , where the norm of $V$
is the order-unit norm $\|\cdot\|_e$, which  is equivalent to the original norm of $V$ by normality of $\Omega$.

Finally, we have
$$\Omega = {\widetilde \Omega}\cap V \subset \{a\in \widetilde V: 0\leq a\}\cap V
 =\{a\in V: 0\leq a\}.$$
Hence the open set $\Omega$   is contained in the interior of the positive cone
$\{a\in V: 0\leq a\}$ in $V$. Conversely, given $v$ in the interior of $\{a\in V: 0\leq a\}$ in $V$,
we have  $v\in \widetilde \Omega \cap V$ by Lemma \ref{1} since each $f\in S_e(\widetilde V)$ restricts
to a positive functional $f|_V \in S_e(V)$ and $f(v) >0$. This proves that $\Omega$ is
the interior of the positive cone $\{a\in V: 0\leq a\}$, equivalently, the closure $\overline\Omega$
is the positive cone of the real C*-algebra $\mathcal{R}(V)$.
\end{proof}

\begin{rem} In view of Theorems \ref{fin}, \ref{c} and \ref{rc},
a Finsler symmetric cone with an orientable extension
 is necessarily normal and linearly homogeneous.
Also, a normal linearly homogeneous Finsler symmetric cone with an oreientation
admits an orientable extension.\\
\end{rem}

\noindent {\bf Acknowledgments}.
This work was partly supported by the Engineering and Physical Sciences Research Council, UK (grant number EP/R044228/1).

\end{document}